\documentclass[12pt]{article}

\usepackage[letterpaper, hmargin=1in, top=1in, bottom=1.2in, footskip=0.7in]{geometry}

\usepackage{rotating,pdflscape,xcolor,amssymb,subfigure,psfrag,amsmath,eufrak,bbm,epsfig,amsthm,mathtools,fancyhdr,graphicx}
\definecolor{vdarkred}{rgb}{0.6,0,0.2}
\definecolor{vdarkblue}{rgb}{0,0.2,0.6}
\usepackage[pdftex, colorlinks, linkcolor=vdarkblue,citecolor=vdarkred]{hyperref}
\usepackage{a4wide,amscd}
\usepackage{tikz} 
\usepackage[usenames,dvipsnames]{pstricks}
\usepackage{pst-grad} 
\usepackage{pst-plot} 

\usepackage[small]{caption}

\newcommand{\lam}{\lambda}

\newcommand{\al}{\alpha}

\newcommand{\bgt}{\begin{itemize}}
	\newcommand{\ent}{\end{itemize}}

\newcommand{\ds}{\displaystyle}

\newcommand{\brem}{\begin{rmk}}
	\newcommand{\erem}{\end{rmk}}
\newcommand{\blem}{\begin{lem}}
	\newcommand{\elem}{\end{lem}}
\newcommand{\bcor}{\begin{cor}}
	\newcommand{\ecor}{\end{cor}}
\newcommand{\bTh}{\begin{Th}}
	\newcommand{\eTh}{\end{Th}}
\newcommand{\bpropo}{\begin{propo}}
	\newcommand{\epropo}{\end{propo}}

\newcommand{\supp}{\operatorname{supp}}

\newcommand{\R}{\mathbb{R}}

\newcommand{\ff}{\frac{1}}

\newcommand{\bbm}{\begin{bmatrix}}
	\newcommand{\ebm}{\end{bmatrix}}
\newcommand{\bes}{\begin{equation*}}
\newcommand{\ees}{\end{equation*}}
\newcommand{\be}{\begin{equation}}
\newcommand{\ee}{\end{equation}}
\newcommand{\beqy}{\begin{eqnarray}}
\newcommand{\eeqy}{\end{eqnarray}}
\newcommand{\beq}{\begin{eqnarray*}}
	\newcommand{\eeq}{\end{eqnarray*}}

\newcommand{\bpm}{\begin{pmatrix}}
	\newcommand{\epm}{\end{pmatrix}}

\newtheorem{Th}{Theorem}[section]

\newtheorem{propo}[Th]{Proposition}

\newtheorem{lem}[Th]{Lemma}

\newtheorem{cor}[Th]{Corollary}

\theoremstyle{definition}
\newtheorem{rmk}[Th]{Remark}

\long\def\symbolfootnote[#1]#2{\begingroup
	\def\thefootnote{\fnsymbol{footnote}}\footnote[#1]{#2}\endgroup} 
\providecommand{\keywords}[1]
{
	\small	
	\textbf{\textit{\footnotesize Keywords ---}} #1
}

\makeatletter
\def\@addpunct#1{%
	\relax\ifhmode
	\ifnum\spacefactor>\@m \else#1\fi
	\fi}
\newcommand{\keywordsname}{Key words}
\def\@setkeywords{%
	{\itshape \keywordsname.}\enspace \@keywords\@addpunct.}
\def\keywords#1{\def\@keywords{#1}}
\let\@keywords=\@empty
\g@addto@macro{\maketitle}{\begingroup%
	\let\@makefnmark\relax  \let\@thefnmark\relax%
	\ifx\@keywords\@mpty\else\@footnotetext{\@setkeywords}\fi%
	\endgroup}
\makeatletter
\def\blfootnote{\gdef\@thefnmark{}\@footnotetext}
\usepackage[affil-it]{authblk}
\date{}
\author{Cambyse Pakzad}
\title{Large deviations principle for the largest eigenvalue of the Gaussian $\beta$-ensemble at high temperature}

\newcommand{\Addresses}{{
		\bigskip
		\footnotesize
		\textsc{MAP 5, UMR CNRS 8145 - Universit\'e Paris Descartes, France}\par\nopagebreak
		\textit{E-mail address} : \texttt{cambyse.pakzad@gmail.com}
	}}
	\keywords{Large deviations principle, Random matrices, Gaussian $\beta$-ensembles, Extreme eigenvalue.}
	
\begin{document}
		\maketitle
		\blfootnote{\textup{2010} \textit{Mathematics Subject Classification}.
			60B20 - 60F10}
	\begin{abstract}We consider the Gaussian $\beta$-ensemble when $\beta$ scales with $n$ the number of particles such that $\ds{{n}^{-1}\ll \beta\ll 1}$. Under a certain regime for $\beta$, we show that the largest particle satisfies a large deviations principle in $\R$ with speed $n\beta$ and explicit rate function. As a consequence, the largest particle converges in probability to $2$, the rightmost point of the semicircle law. 
	\end{abstract}
	\section{Introduction and Main Result}
	The probability field naturally endeavors to reach an understanding of the concept of \textit{rare event}. The theory of large deviations principle merged in an attempt of investigating such idea. Its purpose is to measure the exponential decay rate of probabilities of atypical, extreme, tail events or large fluctuations according to the number of observations. Many stochastic processes can be studied this way and multiple applications have been found in both applied and fundamental mathematics, whose one of the deepest links is with statistical mechanics. On the other hand, large random matrices theory has become a wide area of interest with also many links with other fields, and a natural framework to consider strongly correlated system, namely the eigenvalues. Two statistics are commonly studied: the empirical spectral measure and the largest eigenvalue. The first one is \textit{global} as it involves the whole spectrum, while the second is \textit{local} and is the one we are interested in.

Despite the abundance of matrix models studied and their relative statistics, the collusion with large deviations theory is not well understood as only specific cases are treated. The seminal work on the subject is \cite{betaLDP} where the authors establish a large deviations principle at speed $n^2$ with explicit rate function for the empirical spectral measure of \textit{Gaussian $\beta$-ensemble}, which is a collection of $n$ particles with joint law proportional to: $$P(\mathrm{d}\lambda_1,\ldots ,\mathrm{d}\lambda_n) \propto \exp \left( -n
\sum_{i=1}^{n}V\left( \lambda_i\right) + \beta \sum_{i<j}^{n}\log \left| \lambda_j-\lambda_i\right|
\right) \prod_{i=1}^{n}\mathrm{d}\lambda_i ,$$ for $\beta>0$ and quadratic potential $\ds{V(x)=x^2}$. In a Gibbs interpretation of $P$, the Hamiltonian models a system of $n$ particles at inverse temperature $\beta$ undergoing a Coulomb repulsion force (\textit{log-gas} because the interaction is logarithmic) and an external potential. Thereafter in \cite{cupbook}, the same authors extend the result to a larger class of potential $V$ provided that the confining $V$ grows faster than the repulsion's scale which is logarithmic. By a compactification procedure, the author of \cite{hardy2012note} allows the repulsion and the potential to be of same order. Lifting such restriction, the limiting measure is deterministic and may not have compact support.

Regarding different matrix models, deformed matrices from the Gaussian ensemble have been first studied in \cite{maida2007large} where a large deviations principle at speed $n$ on the largest eigenvalue of the sum of a GOE or GUE drawn matrix with a rank one deterministic symmetric (or Hermitian) is shown. In particular, the question of how large should the perturbation be to untie the largest eigenvalue from the bulk. In a second time, the result is enlarged in \cite{benaych2011fluctuations} to the sum of a finite rank random matrix and the $\beta$-ensemble (its tridiagonal matrix representation is due to \cite{MatrixModelBetaEnsemble}) with potential $V$ satisfying the growth condition of \cite{cupbook}. In another direction, the case of extremes of Wigner matrices without Gaussian tails is conducted in \cite{augeri2016large}.

The large deviations principle for the empirical measure provides a keystone to tackle extreme statistics. Indeed, the largest particle depends directly on positions of the $n-1$ remaining ones. The same authors of \cite{betaLDP} carried on the work in \cite{aging} to a large deviations principle for the extreme eigenvalue at speed $n$ and explicit rate function. They also generalized this result in \cite{cupbook} to further more potentials $V$ assuming a condition on the partition function. In this vein, the large deviations principle proved in \cite{PoissonTempLow} for the empirical spectral measure of the Gaussian $\beta$-ensemble in a regime of high temperature $n\beta =2\gamma \geq 0$ is fundamental to our purpose.

The Gaussian $\beta$-ensemble at high temperature has already been studied in several articles. Regarding the global statistics, the limiting semicircle law and CLT theorem are recovered in \cite{trinh_global} by a martingale approach, see also \cite{trinh_jacobi}, in the regime $n\beta\gg 1$. For local statistics, in \cite{PoissonTempLow}, the bulk eigenvalues are shown to exhibit Poissonian statistics by inspecting the correlation functions under the assumption $n\beta =\gamma\geq 0$. The corresponding result for the edge in the regime $n\beta\ll 1$ is treated by same means in \cite{pakzad}. Later, the bulk result for $n\beta =\gamma\geq 0$ is retrieved by \textit{Minami's method} in \cite{trinh_bulk_poisson} and an extension of the CLT to this new regime is introduced. Besides, the idea of using stochastic operator approach (see \cite{Sutton1, RRV}) in the regime $n\beta\gg 1$ is mentioned in \cite{TWHighTemp}.

The aim of this article is to study by means of large deviations technique the largest eigenvalue of the Gaussian $\beta$-ensemble at high temperature, namely we permit the inverse temperature $\beta$ to depend on $n$ such that $\ds{\beta \ll 1}$ and keep the quadratic potential. Many regimes are available, but we restrict ourselves to $\ds{n\beta\gg 1}$. 

Our proof has the same structure as the one of Theorem 2.6.6 from \cite{cupbook}, the main differences laying in Section \ref{section} where we estimate the tail of the largest eigenvalue and analyze some ratio of partition functions involved in the algebraic computations.

Through this article, we adopt the notation: for $u=u_n$ and $v=v_n$ two sequences, $\ds{u\ll v\iff \frac{u}{v} \xrightarrow[n\infty]{}0}$. We state our main result:
\begin{Th} \label{main_result}
	Let $\beta =\beta_n >0$. Assume $\ds{\frac{\log(n)}{n}\ll \beta \ll \ff{\log(n)}}$. Let $(\lam_1,\ldots,\lam_n)$ be distributed according to $P_{n,\al,\beta}$ with $\ds{\al=\frac{n\beta}{2}\gg 1}$: $$P_{n,\alpha ,\beta }(\mathrm{d}\lambda_1,\ldots ,\mathrm{d}\lambda_n)=
\frac{1}{Z_{n,\alpha ,\beta }} \exp \left( {-\frac{\alpha }{2} \sum_{i=1}^{n}\lambda ^2_i}\right)
\prod_{i<j}^n \left| \lam_i-\lam_j \right| ^\beta \prod_{i=1}^{n}\mathrm{d}\lambda_i.$$ Let $\sigma$ be the semicircle law on $[-2,2]$. Then $\lam_{\max} := \max_{i\leq n}\lam_i$ satisfies a large deviations principle in $\R$ with speed $n\beta$ and rate function 
\[\begin{aligned} J(x)= \left\{ \begin{array}{lll} \displaystyle {- \int \log \left| x-y\right|
\mathrm{d}\sigma (y)+\frac{x^2}{4} -\frac{1}{2}} &{} &{}\quad {\mathrm{if} x\ge 2} \\ +\infty &{}
&{}\quad {\mathrm{if} x< 2.} \end{array} \right. \end{aligned}\]
\end{Th}
As a corollary, we have:
\begin{cor}
	Under the assumptions of Theorem \ref{main_result}, the largest particle \(\displaystyle {\lambda_{\max } :=
\max \nolimits_{i\le n}\lambda_i}\) converges in probability to $2$ as $n$ goes to infinity.
\end{cor} Likewise, the authors of \cite{cupbook,aging,betaLDP} made use of the primal large deviations principle for the empirical spectral measure; we stand on the corresponding result for our setup proved in \cite{PoissonTempLow}:
 \begin{Th}\label{ldp_esd}
 	Let $\beta=\beta_n \ll 1$ such that $n\beta\gg 1$ and let $\ds{\al \sim \frac{n\beta}{2}}$. For $(\lam_1,\ldots,\lam_n)$ $P_{n,\al,\beta}$-distributed, the empirical measure $\ds{L_n=\ff{n}\sum_{i=1}^{n}\delta_{\lam_i} }$ satifies a large deviations principle in $M_1(\R)$ endowed with the weak topology, at speed $n^2\beta$ and with rate function $$I(\mu) = \int_{\R^2} \frac{x^2+y^2}{2}-\ff{2}\log \left|x-y \right|\mathrm{d}\mu^{\otimes 2}(x,y)-\frac{3}{8}.$$ 
 \end{Th}
 \section{Proofs}
 \subsection{Setup}
 	For any $\al,\beta >0$, and $n\geq 1$, we define the partition function: $$Z_{n,\al,\beta}= \int_{\R^n} \exp\left( {-\frac{\al}{2} \sum_{i=1}^{n}\lam^2_i} \right) \left| \Delta_n(\lam)\right| ^\beta \prod_{i=1}^{n}\mathrm{d}\lam_i $$ with the Vandermonde determinant factor: $$  \left| \Delta_n(\lam)\right| ^\beta :=\prod_{i<j}^n \left|x_i-x_j \right|^\beta.$$ 
 	We consider an exchangeable family $(\lam_1,\ldots,\lam_n)$ of random variable with joint law:  \begin{align}\label{beta_ens_joint_law}
 	 P_{n,\al,\beta}(d\lam_1,\ldots,d\lam_n)&= \ff{Z_{n,\al,\beta}} \exp\left( {-\frac{\al}{2} \sum_{i=1}^{n}\lam^2_i}\right)  \left| \Delta_n(\lam)\right| ^\beta \prod_{i=1}^{n}\mathrm{d}\lam_i .
 	\end{align} 
 		In the sequel, we consider $\beta =\beta_n \ll 1$ such that $\ds{{n\beta}\gg 1}$ and omit the subscript $n$ for clarity. The scale parameter $\al$ is set to $\ds{2\al = {n\beta}}$; nonetheless its value may vary according to some context but it will always be notified. The reason of this choice is the following: by the matrix model from \cite{MatrixModelBetaEnsemble}, one can view the particles $(\lam_i)$ distributed according to (\ref{beta_ens_joint_law}) as the eigenvalues of the tridiagonal symmetric random matrix $ H_{n,\al,\beta}$ defined as 
 		$$ \ff{\sqrt{\al}} \begin{pmatrix}
 		g_1 &\ff{\sqrt{2}}X_{n-1} & && \\
 		\ff{\sqrt{2}}X_{n-1}&g_2 & \ff{\sqrt{2}}X_{n-2}&&\\
 		& \ff{\sqrt{2}}X_{n-2}&g_3 & \ff{\sqrt{2}}X_{n-3} &  \\
 		& & \ddots &\ddots & \ddots &    \\
 		& &\qquad  \ddots & \qquad \ddots &\ff{\sqrt{2}}X_{1} &   \\
 		& & &\ff{\sqrt{2}}X_1& g_n 
 		\end{pmatrix} ,$$ with $\ds{(g_i)_{1\leq i \leq n}\sim \mathcal{N}(0,1)}$ i.i.d. sequence, $\ds{(X_i)_{1\leq i \leq n-1}}$ an independent sequence with law $\ds{X_i \sim \chi(i\beta)}$ such that every entries are independent up to symmetry. Using trace invariance, one sees that the empirical spectral measure $L_n$ has asymptotic first moment $0$ and second moment $1$ if $\ds{\al \sim 1+\frac{n\beta}{2}}$.
 		
Regarding the limiting objects of Thereom \ref{main_result}, we introduce: 
$$ \phi(z,\mu)= \int \log \left|z-y \right|\mathrm{d}\mu(y)-\frac{z^2}{4} $$
$$\sigma(x) = \ff{\sqrt{2\pi}} \sqrt{4-x^2}1_{[-2,2]}(x). \qquad \text{(semicircle law)} $$
So that, for $x\geq 2$, \begin{align*}
J(x) &=  -\int \log \left| x-y\right| \mathrm{d}\sigma(y)+\frac{x^2}{4} -\frac{1}{2}
\\&= -\phi(x,\sigma) -\frac{1}{2}.
 \end{align*}
 
 For $\mu \in M_1(\R)$ fixed, the mapping $z\mapsto\phi(z,\mu)$ is decreasing since the quadratic term dominates the logarithm integral. Besides, from \cite[Exercise 2.6.4 p. 81]{cupbook}, we know that $\ds{J\geq 0}$ on $\R$ and $\ds{\min_{x\in \R^+}J(x) = J(2)=0}$.
 
 We also set: $$\phi_n(z,\mu ) =\int \log \left|z-y \right|\mathrm{d}\mu(y)-\frac{n}{n-1}\frac{z^2}{4}.$$
 Note that $\phi_n\leq \phi$ and $\ds{\phi_n \xrightarrow[n\infty]{} \phi=\sup_n \phi_n(z,\mu)}$ uniformly on $[-M,M]\times M_1(\R)$ (and more generally, on $A\times M_1(B)$ with $A,B\subset \R$ and $A$ compact) since $$\left| \phi_n(z,\mu) - \phi(z,\mu)\right| \leq \frac{z^2}{4n} \leq \frac{M^2}{4n}.$$
 \subsection{Strategy}
Following the lines of \cite{aging}, we prove weak large deviations principle along with exponential tightness. The weak large deviations principle is proven if we show: \begin{align*}
-J(x)&= \lim_{\delta \to 0}\limsup_{n\to +\infty}\ff{n\beta}\log P_{n,\al,\beta}\left(\lam_{\max}\in (x-\delta,x+\delta) \right)
\\&=\lim_{\delta \to 0}\liminf_{n\to +\infty}\ff{n\beta}\log P_{n,\al,\beta}\left(\lam_{\max}\in (x-\delta,x+\delta) \right) .
 \end{align*}
 The mapping $\ds{x\mapsto J(x)=-  \int  \log\left| x-y\right|  \mathrm{d}\sigma(y)+\frac{x^2}{4} -\frac{1}{2}}$ is continuous on $(2 ,+\infty)$ and is increasing to $+\infty$ as $x\to +\infty$. Hence it is a good rate function.
 
  Therefore, it is enough to show that (a) for any $x>2$, $$\limsup_{n\infty} \ff{n\beta}\log P_{n,\al,\beta}(\lam_{\max} \geq x)\leq -J(x), $$
 and, b) for any $x>2$, $$\lim_{\delta\to 0} \liminf_{n\infty} \ff{n\beta}\log P_{n,\al,\beta}\big(\lam_{\max} \in (x-\delta,x+\delta)\big)\geq -J(x).$$
 Also we need to check the degenerated case in the definition of $J$, that is: (c) for any $x<2$, $$\limsup_{n\infty} \ff{n\beta}\log P_{n,\al,\beta}(\lam_{\max} \leq x)=-\infty. $$ 
 Indeed, if (a) holds then for $x> 2$, \[\begin{aligned} \lim_{\delta \rightarrow 0} \limsup_{n\infty } \frac{1}{n\beta }\log P_{n,\alpha ,\beta
}&\big (\lambda_{\max } \in (x-\delta ,x+\delta )\big )\\&\le \lim_{\delta \rightarrow 0} \limsup_{n\infty }
\frac{1}{n\beta }\log P_{n,\alpha ,\beta }(\lambda_{\max } \ge x-\delta )\\&\le \lim_{\delta \rightarrow 0}
-J(x-\delta ) = -J(x) . \end{aligned}\]
 Combined with (b) and the case encoded by (c), the weak large deviations principle will follow. 
 
 In order to establish exponential tightness as required, we will show that $$\lim_{M\to +\infty}\limsup_{n\infty} \ff{n\beta}\log P_{n,\al,\beta}(\lam_{\max} > M)=-\infty .$$ Once this proved, the exponential tightness will follow after using c) and the Laplace principle on the following inequality
 	$$\ff{n\beta}\log P_{n,\al,\beta}(\lam_{\max} \notin [x,M]) \leq \ff{n\beta}\log\big( P_{n,\al,\beta}(\lam_{\max} > M )+ P_{n,\al,\beta}(\lam_{\max} < x) \big) .$$
The term $J(x)$ will be introduced through the following \textit{key formula}:
 \begin{align*}\label{key_formula}
 (\star)&:=P_{n,\al,\beta}\left( \lam_i \in A_i,i\in \{1,\ldots,n\}\right)\\&= \ff{Z_{n,\al,\beta}} \int_{\lam_i\in A_i}\mathrm{e}^{-\frac{n\beta}{4} \sum_{i=1}^{n}\lam_i^2 + \beta \sum_{i<j}^{n}\log \left| \lam_i-\lam_j\right| }\mathrm{d}\lam
 \\&=\ff{Z_{n,\al,\beta}} \int_{\lam_i\in A_i}\mathrm{e}^{-\frac{n\beta}{4} \sum_{i=1}^{n-1}\lam_i^2 + \beta \sum_{i<j}^{n-1}\log \left| \lam_i-\lam_j\right|-\frac{n\beta}{4} \lam_n^2 +\beta \sum_{i=1}^{n-1}\log \left| \lam_n-\lam_i\right|  }\mathrm{d}\lam
 \\&=\frac{{Z}_{n-1,\al,\beta}}{Z_{n,\al,\beta}} \int_{\lam_i\in A_i} \mathrm{e}^{-\frac{n\beta}{4} \lam_n^2 +\beta \sum_{i=1}^{n-1}\log \left| \lam_n-\lam_i\right|  }\mathrm{d}\lam_n \mathrm{d}{P}_{n-1,\al,\beta}(\lam_1,\ldots,\lam_{n-1})
 \\&=\frac{{Z}_{n-1,\al,\beta}}{Z_{n,\al,\beta}} \int_{\lam_i\in A_i} \mathrm{e}^{(n-1)\beta  \left( -\frac{n}{n-1}\frac{\lam_n^2}{4}+ \int \log \left| \lam_n-y\right| \mathrm{d}L_{n-1}(y) \right)} \mathrm{d}\lam_n \mathrm{d}{P}_{n-1,\al,\beta}(\lam_1,\ldots,\lam_{n-1})
 \\&=\frac{{Z}_{n-1,\al,\beta}}{Z_{n,\al,\beta}} \int_{\lam_i\in A_i}\mathrm{e}^{(n-1)\beta \phi_n  \left( \lam_n,L_{n-1}\right)} \mathrm{d}\lam_n \mathrm{d}{P}_{n-1,\al,\beta}(\lam_1,\ldots,\lam_{n-1}) .
 \end{align*}
The idea behind is that given $\ds{\left(\lam_1,\ldots,\lam_n \right) \sim P_{n,\al,\beta}}$, the top particle $\lam_{\max}$ can be described in terms of the remaining particles which form the Gaussian $\beta$-ensemble with size $n-1$ and scale parameter $\al$ whose value may differ. According to the description involved, a ratio of partition functions appears in the formula and should be analyzed. Provided the previous equalities, one is prone to summon Theorem \ref{ldp_esd}. 
  \subsection{Estimates}
  \label{section}
In this section are collected some estimates needed in the proof of the main theorem. First, two asymptotics of ratios of partition functions arising from the scale (or potential) perturbation are given. Then, a tail bound for the largest particle is shown. These results allow to precisely estimate asymptotic probabilities required in the definitions of large deviations theory (see \cite{cupbook,dembo2010large}). 

  We begin with a technical but crucial lemma. \begin{lem}\label{technical_lemma}
  	For any $a,b\in\R$ and $\beta>0$, one has $$\left| a+b\right|^\beta \leq 2^\beta \exp\left( \beta \frac{a^2 + b^2}{8}\right)  . $$
  \end{lem}
  \begin{proof}
  	First, recall two inequalities $$\left| x\right| \leq 2\exp\left( {\frac{x^2}{16}}\right)  .$$ $$(x+y)^2\leq 2x^2 + 2y^2. $$
  	Applying the first inequality with $x=a+b$, then using the second one gives
  	\begin{align*}
  	\left| a+b\right|^\beta &\leq \left( 2\mathrm{e}^{\frac{(a+b)^2}{16}} \right)^\beta
  	\leq  \left( 2\mathrm{e}^{\frac{a^2+b^2}{8}} \right)^\beta.
  	\end{align*}
  \end{proof}
  For sake of completeness, we mention the Laplace principle (the proof can be found in \cite{dembo2010large}), a result of practital purpose. It allows to tackle logarithm of sum.
  \begin{lem}
  	(Laplace principle) Let $a_n\xrightarrow[n\infty]{}+\infty$ and a finite number $p$ of nonnegative sequences $b^{(1)}_n,\ldots,b^{(p)}_n$. Then
  	$$\limsup_{n\infty}\ff{a_n} \log \sum_{i=1}^{p}b^{(i)}_n = \max_{i=1,\ldots,p}\limsup_{n\infty}\ff{a_n}\log b^{(i)}_n. $$
  \end{lem}
  
  \begin{rmk}
  	We actually use the following corollary (deduced from $\left\|x\right\|_{\infty}\leq  \left\|x\right\|_{1}  $) : $$\limsup_{n\infty}\ff{a_n} \log \sum_{i=1}^{p}b^{(i)}_n \leq \sum_{i=1}^{p} \limsup_{n\infty}\ff{a_n}\log b^{(i)}_n.$$
  \end{rmk} 
   The next two results are two asymptotics of ratios of partition functions arising from the scale (or potential) perturbation. The second lemma will be used for exponential tightness.
   \begin{lem}\label{partition_limit1}
   	Let $\ds{\al=\frac{n\beta}{2}}$. Assume $\ds{\ff{n}\ll \beta \ll \ff{\log(n)}}$. Then, $$ \frac{Z_{n-1,\al,\beta}}{Z_{n,\al,\beta}}= \ff{2\pi}\exp\left( \frac{1}{2}n\beta \right)  \left( 1+o(1)\right) .$$ \end{lem}
   \begin{proof}
   	Recall the Selberg integral theorem from \cite{cupbook}: for any $\ds{\al,\beta>0}$, $$Z_{n,\al,\beta} = \al^{-\frac{n}{2}-\frac{\beta n (n-1)}{4}}(n!)(2\pi)^{\frac{n}{2}}\prod_{i=0}^{n-1}\frac{\Gamma\big( (j+1)\frac{\beta}{2}\big)}{\Gamma\big( \frac{\beta}{2}\big)}.$$
   	Let $\ds{\al=\frac{n\beta}{2}}$. By Selberg integral, we compute: \begin{align*} \frac{Z_{n-1,\al,\beta}}{Z_{n,\al,\beta}}&= \ff{n\sqrt{2\pi}}\frac{\Gamma(\frac{\beta}{2})}{\Gamma(\frac{n\beta}{2})}\left(\frac{n\beta}{2} \right)^{\ff{2}+\frac{\beta(n-1)}{2}}.
   	\end{align*}
   	From Gamma function asymptotics (near-zero and Stirling approximation), we get 	$$\frac{Z_{n-1,\al,\beta}}{Z_{n,\al,\beta}}= \frac{\sqrt{\frac{n\beta}{2}}}{\sqrt{2\pi}}\frac{2}{n\beta}\frac{\mathrm{e}^{-\frac{n\beta}{2}\log\left( \frac{n\beta}{2}\right)  +\frac{n\beta}{2}}}{\sqrt{2\pi }}\left(\frac{n\beta}{2} \right)^{\ff{2}+\frac{\beta(n-1)}{2}} \left( 1+o(1)\right)  .$$
   	Since $\ds{\beta\log(n\beta)\ll1}$, after cancelation and neglecting small order: $$ \frac{Z_{n-1,\al,\beta}}{Z_{n,\al,\beta}}= \ff{2\pi}\exp\left( \frac{1}{2}n\beta \right)  \left( 1+o(1)\right) .$$
   \end{proof}
   \begin{lem}\label{partition_limit2}
Let $\ds{\al=\frac{n\beta}{2}}$. Assume $\ds{\ff{n}\ll \beta \ll \ff{\log(n)}}$. Then, $$\frac{Z_{n-1,\al-\frac{\beta}{4},\beta}}{Z_{n,\al,\beta}}=\ff{2\pi}\exp\left( \ff{4}+\frac{5}{8}n\beta \right)  \left( 1+o(1)\right) . $$
   \end{lem}
   \begin{proof}
By Selberg integral and Taylor expansion, we compute: \begin{align*} \frac{Z_{n-1,\al-\frac{\beta}{4},\beta}}{Z_{n,\al,\beta}}&= \ff{n\sqrt{2\pi}}\frac{\Gamma(\frac{\beta}{2})}{\Gamma(\frac{n\beta}{2})}\left(\frac{n\beta}{2} \right)^{\ff{2}+\frac{\beta(n-1)}{2}}\left( 1-\ff{2n}\right)^{-\frac{(n-1)}{2}-\frac{\beta (n-1)(n-2)}{4}} 
\\&= \ff{n\sqrt{2\pi}}\frac{\Gamma(\frac{\beta}{2})}{\Gamma(\frac{n\beta}{2})}\left(\frac{n\beta}{2} \right)^{\ff{2}+\frac{\beta(n-1)}{2}}\exp\left(\ff{4}+\frac{n\beta}{8} \right) \left( 1+o(1)\right) .
\end{align*}
Using Gamma function asymptotics, we get $$\frac{Z_{n-1,\al-\frac{\beta}{4},\beta}}{Z_{n,\al,\beta}}= \frac{\sqrt{\frac{n\beta}{2}}}{\sqrt{2\pi}}\frac{2}{n\beta}\frac{\mathrm{e}^{-\frac{n\beta}{2}\log\left( \frac{n\beta}{2}\right)  +\frac{n\beta}{2}}}{\sqrt{2\pi }}\left(\frac{n\beta}{2} \right)^{\ff{2}+\frac{\beta(n-1)}{2}} \exp\left(\ff{4}+\frac{n\beta}{8} \right)\left( 1+o(1)\right) . $$
Since $\ds{\beta\log(n\beta)\ll1}$, after cancelation and neglecting small order: $$\frac{Z_{n-1,\al-\frac{\beta}{4},\beta}}{Z_{n,\al,\beta}}=\ff{2\pi}\exp\left( \ff{4}+\frac{5}{8}n\beta \right)  \left( 1+o(1)\right) . $$
   \end{proof}
   We state a tail bound on the largest particle.
   \begin{lem}\label{tailbound_largest} Let $\al,\beta,t>0$ and $n\geq 2$. Then, $$P_{n,\al,\beta}\left( \left|\lam_1 \right|\geq t\right) \leq \frac{2^{n\beta+\frac{3}{2}}}{\al^{\frac{3}{2}} t} \frac{Z_{n-1,\al-\frac{\beta}{4},\beta}}{Z_{n,\al,\beta}}\exp\left(-\frac{\al }{4}t^2 \right) .$$
   \end{lem}
   \begin{proof} Let $\al,\beta,t>0$. 
   	One has  \begin{align*}
   	P_{n,\al,\beta}\left( \left|\lam_1 \right|\geq t\right) &=	\frac{1}{Z_{n,\al,\beta}}\int_{\left| z_1\right| \geq t }\prod_{j=2}^{n} \left|z_1-z_j \right|^\beta e^{ -\frac{\al}{2}    \lam^2_1  }  \mathrm{d}\lam_1 \times  \\ & \times \int_{\R^{n-1}}\left| \Delta_{n-1}(\lam_2,\ldots,\lam_n)\right| ^\beta  \mathrm{e}^{-\frac{\al}{2} \sum_{i=1}^n  \lam^2_i  } \mathrm{d}\lam_2 \ldots \mathrm{d}\lam_n .
   	\end{align*}
   	Using the bound of Lemma \ref{technical_lemma}, we get \begin{align*}
\prod_{j=2}^{n} \left|z_1-z_j \right|^\beta \mathrm{e}^{ -\frac{\al}{2}    \lam^2_1  } \left| \Delta_{n-1}(\lam_2,\ldots,\lam_n)\right| ^\beta  \mathrm{e}^{-\frac{\al}{2} \sum_{i=1}^n  \lam^2_i  }  \leq &\\ 2^{n\beta}\exp\left( -\ff{2}\left( \al-\frac{\beta}{4} \right) \sum_{i=2}^{n}\lam^2_i -\frac{\al}{4}\lam^2_1\right) .
   	\end{align*}
   	 Assuming that $\ds{\al-\frac{\beta}{4}>0}$, it follows that $$P_{n,\al,\beta}\left( \left|\lam_1 \right|\geq t\right) \leq 2^{n\beta} \frac{Z_{n-1,\al-\frac{\beta}{4},\beta}}{Z_{n,\al,\beta}}\int_{\left| u\right|\geq t}\mathrm{e}^{-\frac{\al u^2}{4}}\mathrm{d}u. $$
   	 By a change of variable,
   	 $$P_{n,\al,\beta}\left( \left|\lam_1 \right|\geq t\right) \leq \frac{2^{n\beta+1}}{\al} \frac{Z_{n-1,\al-\frac{\beta}{4},\beta}}{Z_{n,\al,\beta}}\int_{\left| u\right|\geq \sqrt{\frac{\al}{2}}t}\mathrm{e}^{-\frac{u^2}{2}}\mathrm{s}u. $$
   	The classic Gaussian bound $\ds{\int_y^\infty \exp\left( {-\frac{u^2}{2}}\right) \mathrm{d}u\leq \ff{y}\exp\left( {-\frac{y^2}{2}}\right) }$ yields the result.
   \end{proof}
\subsection{Proof of the Main Result}
Our first result is the following:
\begin{lem}
	Let $\ds{\al=\frac{n\beta}{2}}$. Assume $\ds{\ff{n}\ll \beta\ll 1}$ and $\ds{\limsup_{n\infty}\frac{\log(n)}{n\beta}<+\infty }$, then $$ \lim_{M\to +\infty}\limsup_{n\infty} \ff{n\beta}\log P_{n,\al,\beta}(\lam_{\max} > M)=-\infty. $$
	And, $$\qquad \lim_{M\to +\infty}\limsup_{n\infty} \ff{n\beta}\log P_{n,\al,\beta}(\lam_{\min} < -M)=-\infty. $$
\end{lem}
\begin{rmk}
	The first statement is used to prove exponential tightness as previously mentioned, while the second statement will help for showing a).
\end{rmk}
\begin{proof}
	We begin with the first statement. Let $M>0$. By Lemma \ref{tailbound_largest}, \begin{align*} 
	P_{n,\al,\beta}(\lam_{\max} >M) &=P_{n,\al,\beta}(\exists i\leq n,\lam_i >M)
	\\&\leq \sum_{i=1}^{n} P_{n,\al,\beta}(\lam_i >M) \quad \text{by union bound}
	\\&= nP_{n,\al,\beta}(\lam_1 >M) \quad \text{by exchangeability}
	\\&\leq n\frac{2^{n\beta+\frac{3}{2}}}{\al^{\frac{3}{2}} M} \frac{Z_{n-1,\al-\frac{\beta}{4},\beta}}{Z_{n,\al,\beta}}\exp\left(-\frac{\al }{4}M^2 \right) .
	\end{align*}
	Since we already proved that $\ds{\frac{Z_{n-1,\al-\frac{\beta}{4},\beta}}{Z_{n,\al,\beta}}\sim \ff{2\pi}\exp\left( \ff{4} +\frac{5}{8}n\beta \right) }$ in Lemma \ref{partition_limit2}, the claim follows as soon as $\ds{\limsup_{n\infty}\frac{\log(n)}{n\beta}<+\infty }$. Indeed, if there exists $c<+\infty$ such that $\ds{\limsup_{n\infty}\frac{\log(n)}{n\beta}\leq c}$, then $\ds{\limsup_{n\infty}\ff{n\beta}\log	P_{n,\al,\beta}(\lam_{\max} >M)\leq c+\log(2)+\frac{5}{8}-\frac{M^2}{4}}$.

 One deduces the second statement follows from the first one by noticing that the random vectors $\ds{\left(\lam_1,\ldots,\lam_n \right) }$ and $\ds{\left( -\lam_1,\ldots,-\lam_n \right)}$ have same law.
\end{proof}

\begin{rmk}
The first condition on $\beta$ simply rephrases the regime we focus on, that is $\beta\ll 1$ and $n\beta\gg 1$. The second hypothesis means that either $\ds{0<\limsup_{n\infty}\frac{\log(n)}{n\beta}<+\infty }$ or $\ds{\limsup_{n\infty}\frac{\log(n)}{n\beta}=0 }$. In the sequel, we focus on the latter. It encodes the fact that $\beta$ has slower decay rate than $\ds{\ff{n}}$ by a logarithm factor, that is $\ds{\frac{\log(n)}{n}\ll \beta}$.
\end{rmk}
   We show the first part (a), that is:
   \begin{lem} 	Let $\ds{\al=\frac{n\beta}{2}}$. Assume $\ds{\frac{\log(n)}{n}\ll \beta \ll \ff{\log(n)}}$. For any $x>2$, $$\limsup_{n\infty} \ff{n\beta}\log P_{n,\al,\beta}\left( \lam_{\max} \geq x\right) \leq -J(x).$$
   \end{lem}
   \begin{proof}
   First, note that for $M>x>2$, $$P_{n,\al,\beta}\left( \lam_{\max} \geq x\right) \leq P_{n,\al,\beta}\left( \lam_{\max} \in [x,M]\right) +P_{n,\al,\beta}\left( \lam_{\max} > M\right) . $$
   The second term $\ds{P_{n,\al,\beta}\left( \lam_{\max} > M\right)}$ is negligible as $M\to +\infty$ so after an use of the Laplace principle, we only need to control the quantity $\ds{P_{n,\al,\beta}\left( \lam_{\max} \in [x,M]\right)}$. Likewise, 
   \begin{align*}
   P_{n,\al,\beta}(\lam_{\max} \in [x,M]) &\leq P_{n,\al,\beta}(\lam_{\max} \in [x,M], \lam_{\min} <-M) \\&\qquad + P_{n,\al,\beta}(\lam_{\max} \in [x,M], \lam_{\min} \geq -M)
   \\&\leq P_{n,\al,\beta}(\lam_{\min} <-M)+P_{n,\al,\beta}(\lam_{\max} \in [x,M], \lam_{\min} \geq -M)
   \\&= P_{n,\al,\beta}\left( \lam_{\min} <-M\right)+P_{n,\al,\beta}\left( \exists i, \lam_i \geq x, \forall j, \lam_j \in [-M,M] \right) .
   \end{align*} 
  For the same reasons, we can neglect the first term while upper bounding. The second term is the core and is more handable than the former quantity $\ds{ P_{n,\al,\beta}\left( \lam_{\max} \geq x\right)}$. It can be retranscripted by union bound and exchangeability. It will be linked to $J(x)$.

  We endow the space $ M_1(\R)$ of probability measures on $\R$ with the $L^1$ Wasserstein distance $d$ which metrizes the weak convergence. We note $I_M = [-M,M]^{n-1}$, $\ds{\phi_M =  \log(2M) }$, $B(\delta) = \{\mu \in M_1(\R), d(\mu, \sigma)<\delta \}$ and $B_M(\delta) = \{\mu \in B(\delta), \supp(\mu) \subset [-M,M] \}$. Then for any $x\in [-M,M]$ and $\mu$ supported on $[-M,M]$, one has the upper bound $\ds{\phi(x,\mu)\leq \phi_M}$. 
  
  Let us note $$ \xi:=P_{n,\al,\beta}\left( \exists i, \lam_i \geq x, \forall j, \lam_j \in [-M,M] \right) .$$
  By union bound, exchangeability and \textit{key formula}, one has for $x<M$: \begin{align*}
  \xi&\leq  n\frac{Z_{n-1,\al,\beta}}{Z_{n,\al,\beta}}\int_{x}^M \mathrm{d}\lam_n \int_{I_M} \exp\left( {(n-1)\beta \phi_n (\lam_n,L_{n-1})}\right) \mathrm{d}{P}_{n-1,\al,\beta}(\lam) .\end{align*} Splitting accordingly to the event $\ds{\{ L_{n-1}\in B_M(\delta)\} }$,
   \begin{align*}
   \xi&\leq n\frac{Z_{n-1,\al,\beta}}{Z_{n,\al,\beta}}\int_{x}^M \mathrm{d}\lam_n \int_{I_M} \mathrm{e}^{(n-1)\beta \phi_n (\lam_n,L_{n-1})}  1_{\{L_{n-1}\in B_M(\delta) \}}\mathrm{d}{P}_{n-1,\al,\beta}(\lam)
   \\&\qquad  + n\frac{Z_{n-1,\al,\beta}}{Z_{n,\al,\beta}}\int_{x}^M \mathrm{d}\lam_n \int_{I_M} \mathrm{e}^{(n-1)\beta \phi_n (\lam_n,L_{n-1})}1_{\{L_{n-1}\notin B_M(\delta) \}}\mathrm{d}{P}_{n-1,\al,\beta}(\lam) .
   \end{align*}
   Thus,
   \begin{align*}
   \xi&\leq n\frac{Z_{n-1,\al,\beta}}{Z_{n,\al,\beta}}  \int_{x}^{M}\exp\left( {(n-1)
   	\sup_{\mu \in B_M(\delta)} \phi_n(z,\mu)  }\right) dz + 
   \\&\qquad \qquad \qquad +
   n\frac{Z_{n-1,\al,\beta}}{Z_{n,\al,\beta}} (M-x)\exp\left( {(n-1)\beta\phi_M}\right)  {P}_{n-1,\al,\beta}\left( L_{n-1} \notin B\left( \delta\right) \right)   
   \\ &\leq n\frac{Z_{n-1,\al,\beta}}{Z_{n,\al,\beta}}(M-x)  \exp\left( {(n-1)
   	\sup_{\mu \in B_M(\delta),z\in [x,M]} \phi_n(z,\mu)  } \right) + 
   \\&\qquad \qquad  +n\frac{Z_{n-1,\al,\beta}}{Z_{n,\al,\beta}}(M-x) \exp\left( {(n-1)\beta\phi_M} \right) {P}_{n-1,\al,\beta}\left( L_{n-1} \notin B\left( \delta\right) \right) .    
   \end{align*}
    Hence, for $x<M$ and $\ds{n\beta \ll \log(n)}$, the quantity $\ds{\Lambda := \limsup_{n\infty}\ff{n \beta}\log P_{n,\beta}(\lam_{\max} \in [x,M]) }$ is bounded by:
   \begin{align*}
   \limsup_{n\infty} \ff{n\beta} \log \left(  n\frac{Z_{n-1,\al,\beta}}{Z_{n,\al,\beta}}(M-x)\left(  \mathrm{e}^{(n-1)
   	\sup \phi_n(z,\mu)  } 
   + \mathrm{e}^{(n-1)\beta\phi_M}
    {P}_{n-1,\al,\beta}\left( L_{n-1} \notin B\left( \delta\right) \right)    \right) \right) ,
   \end{align*}
   where the supremum is taken on $\ds{\{\mu \in B_M(\delta),z\in [x,M]\} }$.
   
   Recall that $\ds{ n\beta \gg \log(n)}$ and $\ds{\phi_M:= \log(2M)}$. A double application of Laplace lemma and Lemma \ref{partition_limit1} yields:
   \begin{align*}
\Lambda&\leq 
  \frac{1}{2}+ \limsup_{n\infty}\ff{n\beta}\log \big( \mathrm{e}^{(n-1)\beta\sup_{z,\mu} \phi_n(z,\mu)}\big)  \\&\qquad+ \limsup_{n\infty}\ff{n\beta}\log \big(\mathrm{e}^{(n-1)\beta\phi_M} {P}_{n-1,\al,\beta}\left( L_{n-1} \notin B\left( \delta\right) \right)\big)
   \\&= \frac{1}{2}+\limsup_{n\infty}\ff{n\beta}\log \left(  {P}_{n-1,\al,\beta}\left( L_{n-1} \notin B\left( \delta\right) \right)\right) +\limsup_{n\infty} \sup_{z\in [x,M], \mu \in B_M(\delta)} \phi_n(z,\mu) 
   \\&= \frac{1}{2}+\limsup_{n\infty}\ff{n\beta}\log \left(  {P}_{n-1,\al,\beta}\left( L_{n-1} \notin B\left( \delta\right) \right)\right) + \sup_{z\in [x,M], \mu \in B_M(\delta)} \phi(z,\mu) .
   \end{align*}
   Now, roughly speaking, we show that $\ds{ \mathrm{e}^{(n-1)\beta\phi_M}{P}_{n-1,\al,\beta}\left( L_{n-1} \notin B\left( \delta\right) \right) }$ is approximately $\ds{ \exp\left( -n^2\beta+n\beta\right) }$ when $n$ is large, so that the we can neglect the corresponding term while upper bounding.

   From Theorem \ref{ldp_esd}, we know that the empirical measure $L_{n-1}$ satisfies a large deviations principle in $M_1(\R)$ with speed $n^2\beta$ and rate function $I$, valid for any $\ds{\al \sim \frac{n\beta}{2}\gg 1}$. 
   
   The set $B(\delta)$ is an open ball in $M_1(\R)$ for the distance $d$, and $I$ achieves its unique minimum value $0$ at the semicircle law $\sigma$. Therefore, $$\limsup_{n\infty} \ff{n^2\beta} \log\left(  P_{n-1,\al,\beta}\left( L_{n-1} \notin B(\delta)\right) \right)  \leq -\inf_{B(\delta)^c} I < 0. $$
   
Thus, the term $\ds{ \ff{n\beta} \log\left(  P_{n-1,\al,\beta}\left( L_{n-1} \notin B(\delta)\right) \right)}$ diverges to $-\infty$ as a factor $n$ is changed.
   
   It follows that we can neglect it for the upper bound: $$\Lambda \leq  \frac{1}{2}+\sup_{z\in [x,M], \mu \in B_M(\delta)} \phi(z,\mu).$$
   Since $\Lambda$ is independent of $\delta>0$, we only need to show that: $$\lim\limits_{\delta \to 0} \sup_{z\in [x,M],\mu \in B_M(\delta)} \phi(z,\mu)= \phi(x,\sigma).$$
   It is clear that $$ \liminf_{\delta \to 0} \sup_{z\in [x,M],\mu \in B_M(\delta)} \phi(z,\mu)\geq\sup_{z\in [x,M]} \phi(z,\sigma).$$
   Besides, one can write: $$\phi(z,\mu) = \inf_{\eta>0}  \int  \log\left( \left| z-y\right|\vee \eta\right)  \mathrm{d}\mu(y) -\frac{z^2}{4}=\inf_{\eta>0}\phi_{\eta}(z,\mu).$$ 
   The application $\phi_{\eta}$ is upper semi-continuous on $[-M,M]\times M_1([-M,M])$. Since an infimum of upper semicontinuous functions is also upper semicontinuous, we deduce that $(z,\mu)\mapsto\phi(z,\mu)$ is upper semicontinuous on $[-M,M]\times M_1([-M,M])$. Thus, the reverse inequality holds: $$\ds{\limsup_{\delta \to 0} \sup_{z\in [x,M],\mu \in B_M(\delta)} \phi(z,\mu)\leq\sup_{z\in [x,M]} \phi(z,\sigma)}.$$ 
   
   Now, using the inequality for $x<M$ valid for $M$ large enough: $$  \int  \log\left| x-y\right| \mathrm{d}\sigma(y) - \frac{x^2}{4}  \geq 
   \sup_{z\in [M,+\infty)} \left( \int  \log\left| z-y\right| d\sigma(y) - \frac{z^2}{4}   \right),$$ we get the equality: $$\sup_{z\in [x,M]} \phi(z,\sigma)=\sup_{z\in [x,+\infty(} \phi(z,\sigma)=\phi(x,\sigma).$$ Hence, $$\lim\limits_{\delta \to 0} \sup_{z\in [x,M],\mu \in B_M(\delta)} \phi(z,\mu)= \phi(x,\sigma).$$
      \end{proof}
      Let us show claim (b).
      \begin{lem}
Let $\ds{\al = \frac{n\beta}{2}}$ and assume $\ds{\frac{\log(n)}{n}\ll \beta\ll \ff{\log(n)}}$. For any $x>2$, $$\lim_{\delta\to 0} \liminf_{n\infty} \ff{n\beta}\log P_{n,\al,\beta}\big(\lam_{\max} \in (x-\delta,x+\delta)\big)\geq -J(x).$$
      \end{lem}
      \begin{proof}
    Let $x>2$. Let $\delta>0$ which will be arbitrary close to $0$, without loss of generality, we can assume that $2\delta <x-2$. Fix $r\in (2,x-2\delta)$. Then, with $I_r := (-M,r)^{n-1}$ and $B_{r,M}(\delta):=\{\mu \in B(\delta), \supp \mu \subset [-M,r] \}$, we have 
    \begin{align*}
    \Xi&:=P_{n,\al,\beta}\big(\lam_{\max} \in (x-\delta,x+\delta)\big)\\&\geq P_{n,\al,\beta}\big(\lam_{n} \in (x-\delta,x+\delta),\lam_i \in (-M,r), i =1,\ldots,n-1\big)
    \\ &= \frac{Z_{n-1,\al,\beta}}{Z_{n,\al,\beta}} \int_{x-\delta}^{x+\delta} \mathrm{d}\lam_n \int_{I_r} \mathrm{e}^{(n-1)\beta \phi_n(\lam_n, L_{n-1})}{P}_{n-1,\al,\beta}(\mathrm{d}\lam_1,\ldots,\mathrm{d}\lam_{n-1}) 
    \\ &\geq\frac{Z_{n-1,\al,\beta}}{Z_{n,\al,\beta}} \int_{x-\delta}^{x+\delta} \mathrm{d}\lam_n \int_{I_r} \mathrm{e}^{(n-1)\beta \phi_n(\lam_n, L_{n-1})}1_{\{L_{n-1}\in B_{r,M}(\delta)\} }{P}_{n-1,\al,\beta}(\mathrm{d}\lam) 
    \\&\geq 2\delta \frac{Z_{n-1,\al,\beta}}{Z_{n,\al,\beta}} \mathrm{e}^{  (n-1)\beta \inf_{z\in (x-\delta,x+\delta),\mu \in B_{r,M}(\delta)} \phi_n(z,\mu)}
    {P}_{n-1,\al,\beta}\big(L_{n-1}\in  B_{r,M}(\delta) \big).
    \end{align*}
    Let us show that $\ds{{P}_{n-1,\al,\beta}\big(L_{n-1}\in  B_{r,M}(\delta) \big)}$ converges to $1$ as $n\to +\infty$.
    
    First, one has the following facts: $$\exists i, \lam_i \notin (-M,r) \iff \max \lam_i \geq r \text{ or } \min \lam_i \leq -M$$
    $$\forall x>2, \qquad \limsup_{n\infty} \ff{n\beta}\log P_{n,\al,\beta}(\lam_{\max} \geq x)\leq -J(x) $$
    $$\lim_{M\to +\infty}\limsup_{n\infty} \ff{n\beta }\log P_{n,\al,\beta}(\lam_{\min} < -M)=-\infty. $$
    By union bound, the previous remarks, since $r>2$ and $M>0$ is chosen large enough,
    \begin{align*}
P_{n-1,\al,\beta}\left(L_{n-1}\notin B_{r,M}(\delta) \right) &\leq P_{n-1,\al,\beta} \left( \exists i \in \{1,\ldots,n-1\}, \lam_i \notin (-M,r)\right)
\\&\leq P_{n-1,\al,\beta} \left( \lam_{\max}\geq r \right) +P_{n-1,\al,\beta} \left( \lam_{\min}<-M\right)  \xrightarrow[n\infty]{}0.
    \end{align*}
    Hence, by Lemma \ref{partition_limit1},
    \begin{align*}
\lim\limits_{\delta\to 0 }\liminf_{n\infty} \ff{n\beta}\log P_{n,\al,\beta}\big(\lam_{\max} \in (x-\delta,x+\delta)\big)&\geq \frac{1}{2} +\lim\limits_{\delta \to 0} \inf_{z\in (x-\delta,x+\delta), \mu \in B_{r,M}(\delta)}\phi(z,\mu)
\\& = \frac{1}{2}+\phi(z,\sigma).
    \end{align*}
    The last equality comes from the continuity of $(z,\mu)\mapsto \phi(z,\mu)$ on $[x-\delta,x+\delta]\times M_1([-M,r])$.
\end{proof}
To conclude, we prove the assertion (c) thanks to Theorem \ref{ldp_esd}.
\begin{lem}
Let $\ds{\al = \frac{n\beta}{2}}$ and assume $\ds{\ff{n}\ll \beta\ll 1}$. For any $x<2$, $$\limsup_{n\infty} \ff{n\beta}\log P_{n,\al,\beta}(\lam_{\max} \leq x)=-\infty. $$ 
\end{lem}
\begin{proof}
Let $x<2$. One can build a function $f\in \mathcal{C}_b(\R)$ such that for any $z\leq x$, $f(z)=0$ and $\int f\mathrm{d}\sigma >0$. Consider $\ds{F=\{\mu \in M_1(\R), \int f\mathrm{d}\mu=0\} }$. It is a closed set in $M_1(\R)$ with respect to the weak topology and clearly, $\sigma\notin F$ and $L_n \in F$. Thus, by Theorem \ref{ldp_esd},
\begin{align*}
\limsup_{n\infty}\ff{n^2\beta}\log P_{n,\al,\beta}\left(\lam_{\max} \leq x \right) &\leq \limsup_{n\infty}\ff{n^2\beta}\log P_{n,\al,\beta}\left(L_{n}\left( ]x,2]\right)=0  \right)
\\&\leq \limsup_{n\infty}\ff{n^2\beta}\log P_{n,\al,\beta}\left(L_n\in F \right)<-\inf_F I.
\end{align*} The result follows.
\end{proof}
		\bibliographystyle{plain}
		\bibliography{refldp}

\begin{thebibliography}{10}

\bibitem{TWHighTemp}
Romain Allez and Laure Dumaz.
\newblock Tracy--{W}idom at high temperature.
\newblock {\em J. Stat. Phys.}, 156(6):1146--1183, 2014.

\bibitem{cupbook}
Greg~W. Anderson, Alice Guionnet, and Ofer Zeitouni.
\newblock {\em An introduction to random matrices}, volume 118 of {\em
  Cambridge Studies in Advanced Mathematics}.
\newblock Cambridge University Press, Cambridge, 2010.

\bibitem{augeri2016large}
Fanny Augeri.
\newblock Large deviations principle for the largest {E}igenvalue of {W}igner
  matrices without {G}aussian tails.
\newblock {\em Electron. J. Probab.}, 21:Paper No. 32, 49, 2016.

\bibitem{aging}
G{\'e}rard Ben~Arous, Amir Dembo, and Alice Guionnet.
\newblock Aging of spherical spin glasses.
\newblock {\em Probab. Theory Related Fields}, 120(1):1--67, 2001.

\bibitem{betaLDP}
G{\'e}rard Ben~Arous and Alice Guionnet.
\newblock Large deviations for {W}igner's law and {V}oiculescu's
  non-commutative entropy.
\newblock {\em Probab. Theory Related Fields}, 108(4):517--542, 1997.

\bibitem{benaych2011fluctuations}
Florent Benaych-Georges, Alice Guionnet, and Myl{\`e}ne Ma{\"\i}da.
\newblock Fluctuations of the extreme {E}igenvalues of finite rank deformations
  of random matrices.
\newblock {\em Electron. J. Probab.}, 16:no. 60, 1621--1662, 2011.

\bibitem{PoissonTempLow}
Florent Benaych-Georges and Sandrine P\'ech\'e.
\newblock Poisson statistics for matrix ensembles at large temperature.
\newblock {\em J. Stat. Phys.}, 161(3):633--656, 2015.

\bibitem{dembo2010large}
Amir Dembo and Ofer Zeitouni.
\newblock {\em Large deviations techniques and applications}, volume~38 of {\em
  Stochastic Modelling and Applied Probability}.
\newblock Springer-Verlag, Berlin, 2010.
\newblock Corrected reprint of the second (1998) edition.

\bibitem{MatrixModelBetaEnsemble}
Ioana Dumitriu and Alan Edelman.
\newblock Matrix models for beta ensembles.
\newblock {\em J. Math. Phys.}, 43(11):5830--5847, 2002.

\bibitem{Sutton1}
Alan Edelman and Brian~D. Sutton.
\newblock From random matrices to stochastic operators.
\newblock {\em J. Stat. Phys.}, 127(6):1121--1165, 2007.

\bibitem{hardy2012note}
Adrien Hardy.
\newblock A note on large deviations for 2{D} {C}oulomb gas with weakly
  confining potential.
\newblock {\em Electron. Commun. Probab.}, 17:no. 19, 12, 2012.

\bibitem{maida2007large}
Myl{\`e}ne Ma{\"\i}da.
\newblock Large deviations for the largest {E}igenvalue of rank one
  deformations of {G}aussian ensembles.
\newblock {\em Electron. J. Probab.}, 12:1131--1150, 2007.

\bibitem{trinh_bulk_poisson}
Fumihiko Nakano and Khanh~Duy Trinh.
\newblock Gaussian beta ensembles at high temperature: {E}igenvalue
  fluctuations and bulk statistics.
\newblock {\em Journal of Statistical Physics}, 173(2):295--321, 2018.

\bibitem{pakzad}
Cambyse Pakzad.
\newblock Poisson statistics at the edge of {G}aussian beta-ensembles at high
  temperature.
\newblock {\em arXiv preprint arXiv:1804.08214}, 2018.

\bibitem{RRV}
Jos{\'e}~A. Ram{\'i}rez, Brian Rider, and B{\'a}lint Vir{\'a}g.
\newblock Beta ensembles, stochastic {A}iry spectrum, and a diffusion.
\newblock {\em J. Amer. Math. Soc.}, 24(4):919--944, 2011.

\bibitem{trinh_jacobi}
Tomoyuki Shirai and Khanh~Duy Trinh.
\newblock The mean spectral measures of random {J}acobi matrices related to
  {G}aussian beta ensembles.
\newblock {\em Electron. Commun. Probab.}, 20:13 pp., 2015.

\bibitem{trinh_global}
Khanh~Duy Trinh.
\newblock Global spectrum fluctuations for {G}aussian beta ensembles: A
  martingale approach.
\newblock {\em Journal of Theoretical Probability}, Oct 2017.

\end{thebibliography}
		\Addresses
\end{document}